\documentclass[12pt]{article}

\usepackage{datetime}

\usepackage[all]{xy}

\usepackage{amssymb, amsmath, amsthm,eucal}

\usepackage{mathptmx}
\usepackage[scaled=.90]{helvet}
\usepackage{courier}

\usepackage{graphicx}

\DeclareRobustCommand{\coprod}{\mathop{\text{\fakecoprod}}}
\newcommand{\fakecoprod}{%
  \sbox0{$\prod$}%
  \smash{\raisebox{\dimexpr.9625\depth-\dp0}{\scalebox{1}[-1]{$\prod$}}}%
  \vphantom{$\prod$}%
}

\usepackage[pdftex]{hyperref}

\hypersetup{
  colorlinks = true,
  linkcolor = black,
  urlcolor = blue,
  citecolor = black,
}

\usepackage{amsmath,amssymb,amsthm}

\theoremstyle{plain} 
\newtheorem{thm}{Theorem}
\newtheorem{prop}[thm]{Proposition}
\newtheorem{cor}[thm]{Corollary}

\theoremstyle{definition} 
\newtheorem{example}[thm]{Example}

\numberwithin{equation}{section}
\numberwithin{thm}{section}

\renewcommand{\H}{\mathrm{H}}
\renewcommand{\d}{\mathrm{d}}

\renewcommand{\mod}{\operatorname{mod}}

\newcommand{\GG}{\mathbb{G}}
\newcommand{\TT}{\mathbb{T}}
\newcommand{\WW}{\mathbb{W}}
\newcommand{\ZZ}{\mathbb{Z}}

\newcommand{\Zset}{\mathbb{Z}}
\newcommand{\Rset}{\mathbb{R}}
\newcommand{\Cset}{\mathbb{C}}

\newcommand{\id}{\operatorname{id}}

\newcommand{\RO}{\operatorname{RO}}
\newcommand{\RU}{\operatorname{RU}}

\newcommand{\ind}{\operatorname{ind}}

\newcommand{\Ker}{\operatorname{Ker}}
\newcommand{\Cok}{\operatorname{Cok}}
\newcommand{\Aut}{\operatorname{Aut}}
\newcommand{\Hom}{\operatorname{Hom}}

\newcommand{\SO}{\operatorname{SO}}
\newcommand{\Spinc}{\operatorname{Spin}^c}

\title{\bf Bauer--Furuta invariants\\ and Galois symmetries}

\author{Markus Szymik}

\date{May 2011}

\begin{document}

\maketitle

\begin{abstract}\noindent
  The Bauer--Furuta invariants of smooth 4-manifolds are investigated from a
  functorial point of view. This leads to a definition of equivariant
  Bauer--Furuta invariants for compact Lie group actions. These are
  studied in Galois covering situations. We show that the ordinary
  invariants of all quotients are determined by the equivariant
  invariants of the covering manifold. In the case where the
  Bauer--Furuta invariants can be identified with the Seiberg-Witten
  invariants, this implies relations between the invariants in Galois
  covering situations, and these can be illustrated through elliptic
  surfaces. It is also explained that the equivariant
  Bauer--Furuta invariants potentially contain more information than
  the ordinary invariants.
\end{abstract}

\thispagestyle{empty}

%%%%%%%%%%%%%%%%%%%%%%%%%%%%%%%%%%%%%%%%%%%%%%%%%%%%%%%%%%%%%%%%%%

\section*{\bf Introduction} 

We make a step towards a structural understanding of the Bauer--Furuta
invariants~\cite{Bauer:Refined, BauerFuruta} of smooth 4-manifolds with complex spin
structures, which refine the Seiberg-Witten invariants. The preface
of~\cite{DonaldsonKronheimer} points out the need for such a venture
in the related context of \hbox{Donaldson} invariants. Our
investigations here lead to a definition of an equivariant version of
the Bauer--Furuta invariants in the context of compact Lie group
actions. This can be related to the family version of the Bauer--Furuta
invariants defined in~\cite{Szymik:Families}, as explained in Section~7 of~{\it loc. cit.}. See also~\cite{Szymik:Stable}.

Any attempt to define an equivariant extension of the Seiberg-Witten
invariants will have to face the problem that invariant perturbations
are not generic, that equivariant transversality does not hold in
general, see for example the discussion in~\cite{RuanWang}. These
problems can be circumvented by the homotopical approach. We show how
the equivariant Bauer--Furuta invariants provide insights into Galois
covering situations for any finite Galois group, and illustrate this
by examples for groups of prime order.

Sections 1 and 2 discuss the functoriality of the monopole map and the
resulting definition of equivariant Bauer--Furuta invariants.
See~\cite{Hitchin} for motivation of the use of categorical language
in global differential geometry. Theorem~\ref{thm:functoriality} shows
that the monopole map is functorial on a certain gauge category which
takes the complex spin structures on the 4-manifolds into account.
Theorems~\ref{thm:existence} and~\ref{thm:well-related} concern the
existence of equivariant Bauer--Furuta invariants well-related to the
ordinary ones. Our discussion is kept as detailed and elementary as
possible in order to emphasise the simplicity of the homotopical
approach.

From Section 3 on the focus is on Galois covering situations. If~$G$
is a finite group, and~\hbox{$X\rightarrow X/G$} is a~$G$-covering,
the Bauer--Furuta invariants of~$X/G$ can be recovered from the
equivariant invariants of~$X$, see Theorem~\ref{thm:fixed point
  restriction}. This suggests the definition of a ghost map which
allows the comparison of the~$G$-invariants of~$X$ with the ordinary
invariants of all its quotients~$X/H$ for~\hbox{$H\leqslant G$}. In
nice situations, see Theorem~\ref{thm:hkr}, the ghost map is an
isomorphism away from the order of~$G$.

While Section~4 contains some preparatory calculations, Sections 5
and~6 study the ghost map integrally in the case when the order of~$G$
is a prime number~$p$. In the case when the complex spin structure
on~$X/G$ comes from an almost complex structure, the equivariant
invariants of~$X$ can be computed from the ordinary invariants of~$X$
and~$X/G$, i.e. from Seiberg-Witten invariants, see
Theorem~\ref{thm:relations}, which also says that the latter satisfy a
mod~$p$ congruence. This is illustrated through certain elliptic
surfaces, where these relations are equivalent to congruences between
binomial co-efficients. Finally, there may be more to the equivariant
invariants than just the ordinary invariants of the quotients: by
Theorem~\ref{thm:not injective}, the ghost map is not injective in
general.

All 4-manifolds considered will be closed and oriented. Except where
explicitly mentioned, they will be connected. The first Betti
number~$b^1$ is assumed to vanish. The notation~$e$ and~$s$ will
denote the Euler number and the signature, respectively. The circle
group will be denoted by~$\TT$, and~$L\TT$ will be its Lie algebra.

This work is based on the author's 2002 thesis~\cite{Szymik:Diss}. I would
like to take this opportunity to express my deep gratitude to Stefan
Bauer. I would also like to apologise for the delayed publication; the manuscript had been stalled at two other journals for two years each. In the meantime, the theory developed here has been successfully applied, see for example the work~\cite{Liu+Nakamura, Nakamura1, Nakamura2, Nakamura3} of Liu and Nakamura.

%%%%%%%%%%%%%%%%%%%%%%%%%%%%%%%%%%%%%%%%%%%%%%%%%%%%%%%%%%%%%%%%%%

\section{\bf Functoriality of the monopole map}

In this section, a certain gauge category~$\mathcal G$ is defined such
that the monopole map~-- which will also be reviewed -- is functorial
on it.

\subsection{The categories}

Let~$\mathcal M$ be the following category of manifolds. The objects
are closed oriented 4-manifolds~$X$ with a Riemannian metric. The
metric will usually be omitted from the notation. The orientation and
the metric provide an $\SO(4)$-bundle~$\SO(X)$ of oriented orthonormal
frames over~$X$. A morphism~$f$ from~$X$ to~$Y$ is to be a local
diffeomorphisms for which the differential preserves the orientation
and the metric in each tangent space. Note that such an~$f$ induces an
isomorphism of~$\SO(X)$ with~$f^*\SO(Y)$ over~$X$.

Let~$\mathcal G$ be the following gauge category. The objects are
objects~$X$ of~$\mathcal M$ together with a complex spin
structure~$\sigma_X$ on~$X$, and an Hermitian connection on the
determinant line bundle of~$\sigma_X$. This line bundle will usually
be denoted by~$L(\sigma_X)$, the connection by~$A(\sigma_X)$ or
just~$A$. The reference to the connection will be omitted from the
notation so that~$(X,\sigma_X)$ will be a typical object of~$\mathcal
G$. A complex spin structure may be thought of as
a~$\Spinc(4)$-principal bundle~\hbox{$\Spinc(X,\sigma_X)$} on~$X$ and
an isomorphism
\begin{displaymath}
  \Spinc(X,\sigma_X)\times_{\Spinc(4)}\SO(4)\longrightarrow \SO(X).
\end{displaymath} 
Using this description,
morphisms~$(X,\sigma_X)\rightarrow(Y,\sigma_Y)$ in the gauge category
$\mathcal G$ are pairs~$(f,u)$ where~$f$ is a morphism~$f:X\rightarrow
Y$ in $\mathcal M$ and~$u$ is an isomorphism of the
bundle~$\Spinc(X,\sigma_X)$ with the
pullback~\hbox{$f^*\Spinc(Y,\sigma_Y)$} over~$X$ such that the induced
isomorphism of~$\SO(4)$-bundles is the one coming from~$f$.
Furthermore, the map~$u$ is to be compatible with the connections.

The forgetful functor from the gauge category~$\mathcal G$ to the
category of manifolds~$\mathcal M$ is a fibration. The fibre~$\mathcal
G(X)$ over an object~$X$ of~$\mathcal M$ is by definition the
subcategory of~$\mathcal G$ consisting of those objects which map
to~$X$ and those morphisms which map to $\id_X$.

Given an object~$(X,\sigma_X)$ of the gauge category~$\mathcal G$,
there is an exact sequence
\begin{equation}\label{universal extension}
  1\longrightarrow
  \Aut_{\mathcal G(X)}(X,\sigma_X)\longrightarrow
  \Aut_{\mathcal G}(X,\sigma_X)\longrightarrow
  \Aut_{\mathcal M}(X)
\end{equation}
of groups. The automorphism group~$\Aut_\mathcal M(X)$ is the group of
orientation preserving isometries of~$X$. This is a compact Lie group.
The map from the group~$\Aut_{\mathcal G}(X,\sigma_X)$
to~$\Aut_{\mathcal M}(X)$ need not be surjective, i.e. not every
isometry~$f$ of~$X$ needs to appear in~a morphism $(f,u)$ in~$\mathcal
G$. By definition, this will be the case if and only if~$f^*\sigma_X$
is isomorphic to~$\sigma_X$. On the other hand, if such a~$u$ exists,
it is not determined by~$f$: there are non-identity morphisms covering
the identity of~$X$. These are the elements in the
group~$\Aut_{\mathcal G(X)}(X,\sigma_X)$, which is isomorphic to~$\TT$.

\subsection{The monopole map}

The gauge category~$\mathcal G$ has been constructed so
that the objects~$(X,\sigma_X)$ have all the structure needed to
define the monopole map. The background connection~$A$ gives an
identification of the vector space $\Omega^1(L\TT)$ with the space of
Hermitian connections on~$L(\sigma_X)$ via~\hbox{$a\mapsto A+a$}.
Every such element gives a Dirac operator~$D_{A+a}$ between the two
spinor bundles~$W^\pm(\sigma_X)$. The self-dual part~$F^+_{A+a}$ of
the curvature lives in the space~$\Omega^+(L\TT)$ of
self-dual~2-forms.  The quadratic map from~$\Omega^0(W^+(\sigma_X))$
to that space will be denoted by~\hbox{$\phi\mapsto\phi^2$}. Finally,
let~$\bar\Omega^0(L\TT)$ be the quotient of the space of functions
on~$X$ modulo the constant functions.~(The class of a function will be
denoted by square brackets.) The Seiberg-Witten equations read
\begin{displaymath}
  D_{A+a}(\phi)=0\quad\mathrm{and}\quad F^+_{A+a}=\phi^2
\end{displaymath}
in this context. With the usual gauge fixing, the solutions correspond
to zeros of the monopole map
\begin{eqnarray*}
   \Omega^0(W^+(\sigma_X))\oplus\Omega^1(L\TT) &\longrightarrow&
   \Omega^0(W^-(\sigma_X))\oplus\Omega^+(L\TT)\oplus\bar\Omega^0(L\TT)\\
   (\phi,a)&\longmapsto&(D_{A+a}(\phi),F_{A+a}^+-\phi^2,[\d^*a]).
\end{eqnarray*}
The image of~$(\phi,a)$ decomposes as a sum
\begin{displaymath}
  (0,F^+_A,0)+(D_A(\phi),\d^+a,[\d^*a])+(a\phi,-\phi^2,0).
\end{displaymath}
Thus, the monopole map is polynomial in~$(\phi,a)$: the first summand
is constant, the second summand is linear, it is the
\emph{linearisation} of the monopole map, and the remaining term is
quadratic.

While an easy computation shows that the monopole map
is~$\TT$-equivariant, this will also follow from its functoriality,
which will be proven next.

\subsection{Functoriality}

Let us introduce notation for the characters introduced in the
previous subsection. For every object~$(X,\sigma_X)$ of~$\mathcal G$,
there are vector spaces
\begin{eqnarray*}
  \mathcal U(X,\sigma_X)&=&\Omega_X^0(W^+(\sigma_X))\oplus\Omega_X^1(L\TT)\\
  \mathcal V(X,\sigma_X)&=&\Omega_X^0(W^-(\sigma_X))\oplus\Omega_X^+(L\TT)\oplus\bar\Omega_X^0(L\TT),
\end{eqnarray*}
and the monopole map 
\begin{displaymath}
\mu(X,\sigma_X):\mathcal U(X,\sigma_X)\longrightarrow\mathcal V(X,\sigma_X)
\end{displaymath}
is a map between them. (We shall silently pass to suitable Sobolev
$L^2$-com\-ple\-tions from now on.) It will now be explained how all
this behaves functorially on the gauge category~$\mathcal G$. First of
all, the source and the target will be addressed.

Given a morphism~$(f,u)$ from~$(X,\sigma_X)$ to~$(Y,\sigma_Y)$ in the
gauge category~$\mathcal G$, there are linear maps
\begin{eqnarray}
  (f,u)^*&:&\mathcal{U}(Y,\sigma_Y)\longrightarrow\mathcal{U}(X,\sigma_X)
  \label{isometry 1}\\
  (f,u)^*&:&\mathcal{V}(Y,\sigma_Y)\longrightarrow\mathcal{V}(X,\sigma_X).
  \label{isometry 2}
\end{eqnarray}
The construction is as follows. Every section of the
bundle~$W^{\pm}(\sigma_Y)$ pulls back to give a section of the
pullback bundle~$f^*W^{\pm}(\sigma_Y)$. Under the isomorphism~$u$,
this corresponds to~a section of~$W^{\pm}(\sigma_X)$.  This gives
maps~\hbox{$\Omega^0_X(W^{\pm}_X)\leftarrow\Omega^0_Y(W^{\pm}_Y)$}.
Functions can be pulled back along~$f$, and constant functions pull
back to constant functions. This describes~a
map~\hbox{$\bar\Omega^0_X(L\TT)\leftarrow\bar\Omega^0_Y(L\TT)$}.
Similarly, the usual pullback of 1-forms yields~a map~\hbox{$\Omega^1_X(L\TT)\leftarrow\Omega^1_Y(L\TT)$}. Finally,~another map~\hbox{$\Omega^+_X(L\TT)\leftarrow\Omega^+_Y(L\TT)$} is induced by pulling
back 2-forms: all there is left to remark is that the pullback of~a
self-dual 2-form is self-dual as well.

This finishes the description of the maps~$(f,u)^*$. They are almost
isometries in the sense that the norms are preserved up to a factor:
the degree of~$f$. As functoriality of the maps $(f,u)^*$ is easy to
check, the following proposition summarises the discussion.

\begin{prop}\label{prop:functor_into_hilbert}
  Given a morphism~$(f,u)$ from~$(X,\sigma_X)$ to~$(Y,\sigma_Y)$ in
  the gauge category~$\mathcal G$, the isometries
  {\upshape(\ref{isometry 1})} and {\upshape(\ref{isometry 2})}
  make~$\mathcal U$ and~$\mathcal V$ into contravariant functors
  from~$\mathcal G$ to the category of Hilbert spaces and continuous
  linear maps.
\end{prop}

Note that, as a special case, for any object~$(X,\sigma_X)$ of the
gauge category~$\mathcal{G}$, there is an action of the
group~$\Aut_{\mathcal G}(X,\sigma_X)$ on the Hilbert spaces
$\mathcal{U}(X,\sigma_X)$ and~$\mathcal{V}(X,\sigma_X)$ by isometries.
The subgroup~\hbox{$\Aut_{\mathcal G(X)}(X,\sigma_X)\cong\TT$} acts by
scalar multiplication on the sections of the complex bundles and
trivially on the other ones.

Now that its source and target have been dealt with, the monopole map
itself will be addressed. From the maps defined above one can build an
obvious diagram, and the following proposition states that it
commutes.  The proof can be left to the reader.

\begin{prop}  
  Consider~a morphism from~$(X,\sigma_X)$ to~$(Y,\sigma_Y)$ in the gauge
  category~$\mathcal G$. If one uses the maps defined in the previous
  proposition as vertical arrows, the diagram
  \begin{center}
    \mbox{
      \xymatrix@C=50pt{
        \mathcal{U}(X,\sigma_X)\ar[r]^{\mu(X,\sigma_X)}&
        \mathcal{V}(X,\sigma_X) \\
        \mathcal{U}(Y,\sigma_Y)\ar[r]^{\mu(Y,\sigma_Y)}\ar[u]& 
        \mathcal{V}(Y,\sigma_Y). \ar[u] 
      } 
    }
  \end{center}
  commutes.
\end{prop}

As a special case again, for any object~$(X,\sigma_X)$ of the gauge
category~$\mathcal G$, the monopole map~$\mu(X,\sigma_X)$
is~$\Aut_{\mathcal G}(X,\sigma_X)$-equivariant. In particular, it
always is~$\TT$-equivariant.

The previous proposition might suggest the question whether or not the
collection of the monopole maps is~a natural transformation between
the functors~$\mathcal{U}$ and~$\mathcal{V}$. But this would ignore
the fact that the monopole maps are of~a different nature than the
linear maps of Proposition~\ref{prop:functor_into_hilbert}. A better
way of describing the situation is as follows. Let~$\mathcal F$ be the
following category of non-linear Fredholm maps. An object is~a
continuous map~$\mu$ between Hilbert spaces~$\mathcal U$ and~$\mathcal
V$ which satisfies the conditions of the construction
in~\cite{BauerFuruta}.  The morphisms from~\hbox{$\mu_1:\mathcal
  U_1\rightarrow\mathcal V_1$} to~\hbox{$\mu_2:\mathcal
  U_2\rightarrow\mathcal V_2$} are pairs of continuous linear
maps~\hbox{$\mathcal U_1\rightarrow\mathcal U_2$} and~\hbox{$\mathcal
  V_1\rightarrow\mathcal V_2$} between Hilbert spaces such that the
evident diagram commutes. The two preceding propositions imply the
following.

\begin{thm}\label{thm:functoriality}
  The monopole maps gives rise to~a contravariant functor from the
  gauge category~$\mathcal G$ to the category~$\mathcal F$ of
  non-linear Fredholm maps.
\end{thm}

Composition with the construction in \cite{BauerFuruta}, which is
functorial as well, gives~a functor into~a category of stable homotopy
classes of maps. This will be explained in the next section.

%%%%%%%%%%%%%%%%%%%%%%%%%%%%%%%%%%%%%%%%%%%%%%%%%%%%%%%%%%%%%%%%%%

\section{\bf Equivariant Bauer--Furuta invariants}

In this section, after a brief review of the ordinary Bauer--Furuta
invariants, an equivariant extension of them is defined
and commented on.

The notation from equivariant stable homotopy theory used here will be
fairly standard, see~\cite{LMS} for example. For a compact Lie
group~$\GG$ and a~$\GG$-repre\-sen\-ta\-tion~$V$, the
one-point-compactification will be denoted by~$S^V$. Given two finite
pointed $\GG$-CW-complexes~$M$ and~$N$, the group of~$\GG$-equivariant
stable maps from~$M$ to~$N$ with respect to a~$\GG$-universe~$\mathcal
V$ will be denoted by~$[M,N]^\GG_\mathcal V$. If the universe is
understood, the reference to it will be omitted.

\subsection{Ordinary Bauer--Furuta invariants}

The definition of the invariants from the monopole map rests on the
following construction. Let~$\mathcal{U}$ and~$\mathcal{V}$ be two
Hilbert spaces on which~$\GG$ acts via isometries. One may assume
that~$\mathcal{V}$ is~a $\GG$-universe, although that is not strictly
necessary. Let~$\mu$ be~a~$\GG$-map from~$\mathcal{U}$
to~$\mathcal{V}$ which admits~a
decomposition~\hbox{$\mu=\lambda+\kappa$} into~a sum of~a linear
map~$\lambda$ which has finite-dimensional kernel and cokernel, and~a
continuous map~$\kappa$ which maps bounded sets into compact sets.
Furthermore, it will be required that pre-images under~$\mu$ of
bounded sets are bounded. Let~$\Cok(\lambda)$ denote the orthogonal
complement of the image of~$\lambda$ in~$\mathcal{V}$. In this
situation,~an equivariant stable homotopy class
in~\hbox{$[S^{\Ker(\lambda)},S^{\Cok(\lambda)}]^{\GG}_{\mathcal{V}}$}
can be defined, see~\cite{BauerFuruta}.

The monopole maps fit into this framework, such that this construction
can be used to define invariants. Let~$(X,\sigma_X)$ denote an object
of~$\mathcal G$. The source~\hbox{$\mathcal U=\mathcal U(X,\sigma_X)$}
and the target~\hbox{$\mathcal V=\mathcal V(X,\sigma_X)$} of the monopole map
are~$\TT$-universes, containing only representations isomorphic
to~$\Rset$ or~$\Cset$. It is shown in~\cite{BauerFuruta} that the
monopole map~$\mu=\mu(X,\sigma_X)$ and its
linearisation~\hbox{$\lambda=\lambda(X,\sigma_X)$} satisfy the
assumptions required by the construction mentioned in the preceding
paragraph. This allows to define an element~$m(X,\sigma_X)$ in
$[S^{\Ker(\lambda)},S^{\Cok(\lambda)}]^\TT_\mathcal V,$ the
\textit{Bauer--Furuta invariant} of~$(X,\sigma_X)$. It is independent
of the metric and the reference connection.

\subsection{A first application of functoriality}

Let~$(X,\sigma_X)$ be an object of the gauge category~$\mathcal G$.
Let a compact Lie group~$G$ act on~$X$ preserving the orientation. One
may assume that $G$ also preserves the metric. Thus, if the action is
faithful, the group~$G$ is a subgroup of~$\Aut_\mathcal M(X)$. The
complex spin structure is called~\textit{$G$-invariant} if for all~$g$
in~$G$ the complex spin structure~$g^*\sigma_X$ is isomorphic
to~$\sigma_X$. (This is the case if and only if~$G$ is in the image of
the rightmost map in the exact sequence (\ref{universal extension}).)
In this case, there is even an
isomorphism~\hbox{$g^*\sigma_X\cong\sigma_X$} which respects the
reference connection.

If~$\sigma_X$ is a~$G$-invariant complex spin structure on~$X$, the exact
sequence (\ref{universal extension}) gives rise to an extension
\begin{equation}\label{GG extension}
  1\longrightarrow\TT\longrightarrow\GG\longrightarrow G\longrightarrow1
\end{equation}
of~$G$ by~$\TT$. Since~$\GG$ is a subgroup of~$\Aut_{\mathcal
  G}(X,\sigma_X)$, functoriality implies the existence
of~a~$\GG$-action on the source and the target of the monopole map
such that it will not only be~$\TT$-equivariant but in
fact~$\GG$-equivariant. The construction in~\cite{BauerFuruta} then
allows to define an equivariant invariant.

\begin{thm}\label{thm:existence}
  If a compact Lie group~$G$ acts on~$X$, and if the complex spin
  structure~$\sigma_X$ is~$G$-invariant, the monopole map defines an
  equivariant invariant~$m_G(X,\sigma_X)$ in the
  group~$[S^{\Ker(\lambda)},S^{\Cok(\lambda)}]^{\GG}$.
\end{thm}

The class~$m_G(X,\sigma_X)$ does not depend on the~$G$-invariant
metric, as in the non-equivariant case. Neither matters the reference
connection. As in the non-equivariant case, the universe in question
should be~$\mathcal{V}$. While this need not be~a~$\GG$-universe in
the first place, it will be after suitably enlarging the source and
the target of the monopole map with~$\GG$-representations on which the
map is defined to be the identity.~A situation in which this
correction is not necessary is that of free actions of finite groups,
which is the subject of the following sections.

\subsection{Forgetful maps}

As soon as equivariant invariants have been defined, one may discuss
the amount of information they contain -- as compared to the ordinary
invariants. It is obvious from the definition that one gets the
ordinary invariant~$m(X,\sigma_X)$ back from the equivariant
invariant~$m_G(X,\sigma_X)$ by forgetting the equivariance coming from
the non-trivial elements of~$G$.

\begin{thm}\label{thm:well-related}
  The forgetful map 
  \begin{displaymath}
    [S^{\Ker(\lambda)},S^{\Cok(\lambda)}]^{\GG} \longrightarrow
    [S^{\Ker(\lambda)},S^{\Cok(\lambda)}]^{\TT}    
  \end{displaymath}
  sends the equivariant invariant $m_G(X,\sigma_X)$ to the ordinary
  Bauer--Furuta invariant~$m(X,\sigma_X)$.
\end{thm}

While this is a trivial observation, it is important from a
structural point of view, since it shows that the equivariant version
of the invariant is really an extension of the ordinary invariant.

\subsection{Two elementary examples}\label{sec:examples}

Given any 4-manifold~$X$, there are two elementary ways of producing a
4-manifold with~$G$-action. On the one hand,~$X$ itself carries the trivial
action. On the other hand, if~$G$ is finite, the group~$G$ acts on the
(non-connected) 4-manifold~$G\times X$ by permuting the components. The
equivariant invariants for these two cases will be discussed now.

In these two examples, the complex spin structure~$\sigma_X$ on~$X$
will not only be~$G$-invariant, but even {\it~$G$-equivariant}. This
means that the short exact sequence~(\ref{GG extension}) splits and
that a splitting has been chosen: $G$ is a subgroup of the
group~$\Aut_{\mathcal G}(X,\sigma_X)$.

\begin{example}
  If~$G$ acts trivially on~$X$, every complex spin
  structure~$\sigma_X$ is~$G$-invariant. One can also endow~$\sigma_X$
  with the trivial $G$-action. As a consequence, the group~$\GG$ can
  be identified with the product $\TT\times G$. Since~$G$ acts
  trivially on the source and the target of the monopole map, the
  forgetful map can be split by the map which is the identity on
  representatives. In particular, the equivariant
  invariant~$m_G(X,\sigma_X)$ is just the ordinary
  invariant~\hbox{$m(X,\sigma_X)$}, regarded as~a~$\GG$-map.
\end{example}

\begin{example}\label{ex:2}
  Let now~$G$ be finite. If one chooses a complex spin
  structure~$\sigma_X$ on~$X$, this gives a complex spin
  structure~$G\times\sigma_X$ on $G\times X$. There is an
  obvious~$G$-action on~$G\times\sigma_X$, so that the symmetry
  group~$\GG$ can be identified with the product~$\TT\times G$ as
  above.  The invariant of a sum is the smash product of the
  invariants of the summands, see~\cite{Bauer:Sum}. It follows for the
  (non-connected)~4-manifold
  \begin{displaymath}
    (G\times X,G\times\sigma_X)=\coprod_{g\in G}(X,\sigma_X)
  \end{displaymath}
  that the ordinary invariant~$m(G\times X,G\times\sigma_X)$ is the
  smash product
  \begin{equation}\label{first_smash}
    \bigwedge_{g\in G}m(X,\sigma_X)\in \biggl[\;\bigwedge_{g\in
      G}S^{\Ker(\lambda)},\bigwedge_{g\in G}S^{\Cok(\lambda)}\;\biggr]^{\TT}.
  \end{equation}
  The functor~$\wedge_{g\in G}$ really takes values in~$(\TT\times
  G)$-spaces and \hbox{$(\TT\times
    G)$}-maps. See~\cite{tomDieck:Trans} for the construction and its
  properties. For example, as~$m(X,\sigma_X)$ is~a map between
  compactified representations, the map~(\ref{first_smash}) is a map
  between the induced representations:
  \begin{displaymath}
    \bigwedge_{g\in G}S^{\Ker(\lambda)}\cong S^{\ind_1^G(\Ker(\lambda))}
  \end{displaymath}
  and similarly for~$\Cok(\lambda)$. Since~$G$ acts on~$G\times X$ by
  permuting the factors, the equivariant invariant~$m_G(G\times
  X,G\times\sigma_X)$ is also~given by~(\ref{first_smash}), but
  considered as a~$(\TT\times G)$-map.
\end{example}

The two examples have a common flavour. They show how to relate
certain constructions in equivariant stable homotopy theory to
constructions of~4-manifolds. The main theorem in \cite{Bauer:Sum} has
the same flavour, but is much deeper: it relates the smash product (or
the composition) to connected sums. Here, it has been shown how, for a
stable~$\TT$-map~$f$ which is realised by~a~4-manifold, also the
stable~$(\TT\times G)$-maps~$f$ and -- if~$G$ is finite --
also~$\wedge_{g\in G}f$ can be realised.

\subsection{Reduction mod~\texorpdfstring{$\TT$}{T}}

As for the structure of the
groups
\begin{displaymath}
  [S^{\Ker(\lambda)},S^{\Cok(\lambda)}]^{\GG}
\end{displaymath}
in which the
equivariant invariants live, in nice situations one can pass to an
isomorphic group of $G$-equivariant maps. For that, it will be
necessary to assume that certain positivity conditions are satisfied.
On the one hand, there is to be an actual $\GG$-representation~$V$
such that
\begin{equation}\label{eq:first_condition}
  [V]=\ind_{\GG}(D_A)
\end{equation}
in~$\RO(\GG)$. (In particular, the index must not have negative
dimension. But this is not sufficient.) The projective space~$\Cset
P(V)$ is~a~$G$-space, and the notation~$\Cset P(V)_+$ will be used
for~$\Cset P(V)$ with a disjoint $G$-fixed base-point added. Let us
write~$W$ for the~$G$-representation $\H^+(X)$. On the other hand, it
will have to be assumed that
\begin{equation}\label{eq:second_condition}
  \dim_\Rset(W^{G})\geqslant 2
\end{equation}
holds. In particular, one may choose~a complement~$W-1$ of~a trivial
subrepresentation in~$W$.  The following can be proven as
in~\cite{BauerFuruta}, taking care of the additional $G$-action.

\begin{prop}\label{reduction}
  Under the two conditions {\upshape (\ref{eq:first_condition})} and
  {\upshape (\ref{eq:second_condition})}, there is an isomorphism
  \begin{displaymath}
    [S^{\Ker(\lambda)},S^{\Cok(\lambda)}]^{\GG} \cong [\Cset P(V)_+,S^{W-1}]^G
  \end{displaymath}
  of groups.
\end{prop}

The requirement~(\ref{eq:first_condition}) will be met in relevant
situations. On the one hand, geo\-me\-try can come for help: For
example, if~$G$ acts on a complex surface~$X$ via holomorphic maps,
kernel and cokernel of the Dirac operator can be interpreted and
(ideal\-ly) computed using coherent cohomology. If the cokernel
vanishes, the kernel serves as~$V$. On the other hand, if~$G$ is
finite, and~$X$ is a~$G$-Galois covering of~$Y$, and~$\sigma_Y$ has
non-negative index, then~(\ref{eq:first_condition}) is satisfied for
the pullback of~$\sigma_Y$ to~$X$. This will be the situation to which
we turn next.

%%%%%%%%%%%%%%%%%%%%%%%%%%%%%%%%%%%%%%%%%%%%%%%%%%%%%%%%%%%%%%%%%%

\section{\bf Galois symmetries}

In the previous section, an invariant has been defined for
4-manifolds~$X$ with an action of a compact Lie group~$G$ and
a~$G$-invariant complex spin structure~$\sigma_X$. From this section
on, the group~$G$ will be finite and act freely on~$X$, adding the
aspect that the quotient~$X/G$ is a 4-manifold as well. This suggests
to compare the invariants of~$X$ with those of~$X/G$. The problem to
face is that there need not be a compatible complex spin structure on
the quotient; and if there is, it need not be unique. The following
proposition clarifies the situation.

\begin{prop} 
  Let~$\sigma_X$ be a~$G$-invariant complex spin structure on~$X$, and
  let~$\GG$ be the corresponding extension of~$G$ by~$\TT$. For any
  subgroup~$H$ of~$G$ there are canonical bijections between the
  following sets.
  \begin{list}{}{\leftmargin20pt\labelwidth20pt\itemsep5pt\topsep5pt}
  \item[\upshape(1)] The set of~$H$-actions~$j$ on~$\sigma_X$
    compatible with the action of~$H$ on~$X$.
  \item[\upshape(2)] The set of subgroups~$H(j)$ of~$\GG$ which map
    isomorphically to~$H$ under the projection from~$\GG$ to~$G$.
  \item[\upshape(3)] The set of isomorphism classes of complex spin
    structures~$\sigma_{X/H}(j)$ on~$X/H$ such that the pullback along
    the quotient map~\hbox{$q:X\rightarrow X/H$} is isomorphic
    to~$\sigma_X$.
  \end{list}
\end{prop}

\begin{proof}
  The pullbacks of the complex spin structures have an obvious action.
  Conversely, given an action, one may pass to the quotient. This
  gives the bijection between the sets in~(1) and~(3). The bijection
  between the sets in~(1) and~(2) is given by the fact that actions
  on~$\sigma_X$ correspond to group homomorphism into the automorphism
  group, which is~$\GG$ in this case.
\end{proof}
  
The notation~$J_H(\sigma_X)$ will be used for any of the sets from the
proposition.

\subsection{A second application of functoriality}

Let us fix a subgroup~$H$ and an element~$j$ in~$J_H(\sigma_X)$.
Let~$\sigma_{X/H}$ be the induced complex spin structure on~$X/H$. The
first aim of this section is the identification of the monopole map~--
and therefore the invariants -- for the pair~$(X/H,\sigma_{X/H})$ with
fixed point data of $(X,\sigma_X)$. Note that there is still some
symmetry downstairs on~$X/H$ which can be taken into account: if~$WH$
is the Weyl group of~$H$ in~$G$, then~$WH$ acts on~$X/H$ and
leaves~$\sigma_{X/H}$ invariant. The relevant extension of~$WH$
by~$\TT$ is given by the Weyl group~$\WW H(j)$ of~$H(j)$ in~$\GG$.

Note that the quotient map~$q:X\rightarrow X/H$ defines a morphism in
the category~$\mathcal G$ from~$(X,\sigma_X)$ to~$(X/H,\sigma_{X/H})$.
It has been shown in the previous section that this leads to a
commutative diagram
\begin{center}
  \mbox{
    \xymatrix@C=80pt{%
      \mathcal{U}(X,\sigma_X)\ar[r]^{\mu(X,\sigma_X)}
      &\mathcal{V}(X,\sigma_X) \\
      \mathcal{U}(X/H,\sigma_{X/H}(j))\ar[r]^{\mu(X/H,\sigma_{X/H}(j))}
      \ar[u]^{q^*} 
      & \mathcal{V}(X/H,\sigma_{X/H}(j)).\ar[u]_{q^*}
    }
  }
\end{center}
As a consequence of functoriality, the images of the vertical arrows
live in the~$H(j)$-fixed subspaces. Thus there is a commutative
diagram as above with the top arrow replaced by its restriction to
the~$H(j)$-fixed points. In a Galois covering situation as at hand,
if~$E_{X/H}$ is a vector bundle on~$X/H$ and~$E_X$ is the pullback,
the pullbacks of the sections of~$E_{X/H}$ are exactly
the~$H$-invariant sections of~$E_X$. Thus, with this adaption of the
target, the vertical arrows become isomorphisms:

\begin{prop}
  The vertical arrows in the diagram
  \begin{center}\mbox{
      \xymatrix@C=80pt{%
        \mathcal{U}(X,\sigma_X)^{H(j)}\ar[r]^{\mu(X,\sigma_X)^{H(j)}}
        &\mathcal{V}(X,\sigma_X)^{H(j)} \\
        \mathcal{U}(X/H,\sigma_{X/H}(j))\ar[r]^{\mu(X/H,\sigma_{X/H}(j))}
        \ar[u]^{q^*}
        &\mathcal{V}(X/H,\sigma_{X/H}(j))\ar[u]_{q^*} }}
  \end{center}
  are isomorphisms.
\end{prop}

This proposition identifies the monopole
map~$\mu(X/H,\sigma_{X/H}(j))$ of the quotient with the restriction
of~$\mu(X,\sigma_X)$ to the~$H(j)$-fixed points. One may use the
isomorphisms given by~$q^*$ as in the proposition above to get an
isomorphism
\begin{equation}\label{fixed_point_iso}
  [S^{\Ker(\lambda_{X/H})},S^{\Cok(\lambda_{X/H})}]^{\WW H(j)}
  \stackrel{\cong}{\longrightarrow}
           [S^{\Ker(\lambda_X)^{H(j)}},S^{\Cok(\lambda_X)^{H(j)}}]^{\WW
             H(j)}
\end{equation}
of groups. The previous proposition implies the following structural
result on the Bauer--Furuta invariants.

\begin{thm}\label{thm:fixed point restriction}
  The~$H(j)$-fixed point map
  \begin{equation}
    [S^{\Ker(\lambda_X)},S^{\Cok(\lambda_X)}]^{\GG}
    \longrightarrow
        [S^{\Ker(\lambda_X)^{H(j)}},S^{\Cok(\lambda_X)^{H(j)}}]^{\WW
          H(j)}
  \end{equation}
  sends the equivariant invariant~$m_G(X,\sigma_X)$ of $X$ to a class
  which can be identified with the
  invariant~$m_{WH}(X/H,\sigma_{X/H}(j))$ of the quotient~$X/H$ by
  means of the isomorphism~{\upshape (\ref{fixed_point_iso})}.
\end{thm}  

This theorem has an interesting consequence. Given a 4-manifold~$Y$
with finite fundamental group~$G$, let~$X$ be a universal covering
of~$Y$. This will be a Galois covering of~$Y$ with group~$G$. For each
complex spin structure~$\sigma_Y$ on~$Y$, the pullback~$\sigma_X$ of
this to~$X$ will be $G$-equivariant. The theorem above implies
that~$m(Y,\sigma_Y)$ can be obtained from the equivariant
invariant~$m_G(X,\sigma_X)$ as the restriction to the $G$-fixed
points.

\begin{cor}\label{finite fundamental groups}
  The information of the ordinary Bauer--Furuta invariants of~4-manifolds
  with finite fundamental group is contained in the equivariant Bauer--Furuta
  invariants of their universal coverings.
\end{cor}

\subsection{The ghost map -- first version}

In Corollary \ref{finite fundamental groups}, the point of view of~$Y$
has been emphasised, answering the question how invariants
of~$(Y,\sigma_Y)$ can be interpreted in terms of invariants
of~$(X,\sigma_X)$. Now one may also take the point of view of~$X$ and
try to understand the information which is contained in its
equivariant invariants.

Up to this point, maps out of the group~$[S^{\Ker(\lambda_X)},
S^{\Cok(\lambda_X)}]^{\GG}$ in two directions have been considered. On
the one hand, in the foregoing section, it has been noted that the
forgetful map
\begin{displaymath}
  [S^{\Ker(\lambda_X)},
    S^{\Cok(\lambda_X)}]^{\GG} \longrightarrow
  [S^{\Ker(\lambda_X)},
    S^{\Cok(\lambda_X)}]^{\TT}
\end{displaymath}
maps~$m_G(X,\sigma_X)$ to~$m(X,\sigma_X)$. In fact, the image of the
forgetful map lies in the $G$-invariants, so that we might even
consider this as map
\begin{equation}\label{invariant_map}
  [S^{\Ker(\lambda_X)}, S^{\Cok(\lambda_X)}]^{\GG} \longrightarrow \H^0(G;
    [S^{\Ker(\lambda_X)}, S^{\Cok(\lambda_X)}]^{\TT}).
\end{equation}
On the other hand, in this section, the relevance of the fixed point
maps has been shown. These maps took the form
\begin{displaymath}
  [S^{\Ker(\lambda_X)}, S^{\Cok(\lambda_X)}]^{\GG}
  \longrightarrow [S^{\Ker(\lambda_X)^{H(j)}},
    S^{\Cok(\lambda_X)^{H(j)}}]^{\WW H(j)}
\end{displaymath}
for a subgroup~$H$ of~$G$ and an element~$j$ in~$J_H(\sigma_X)$. They can be
composed with a forgetful map like (\ref{invariant_map}) to get a map
\begin{displaymath}
  [S^{\Ker(\lambda_X)}, S^{\Cok(\lambda_X)}]^{\GG}
  \longrightarrow \H^0(WH; [S^{\Ker(\lambda_X)^{H(j)}},
    S^{\Cok(\lambda_X)^{H(j)}}]^{\TT}).
\end{displaymath}
The image of the element~$m_G(X,\sigma_X)$ under this composition has
be identified with the ordinary invariant~$m(X/H,\sigma_{X/H}(j))$ of
the quotient~$X/H$ with respect to the complex spin structure
corresponding to~$j$. Notice that the map~(\ref{invariant_map}) is
just the case where~$H$ is the trivial subgroup.

Let us now put all this information together in one map. One can sum over
the~$j$ in~$J_H(\sigma_X)$ to obtain a map
\begin{displaymath}
  [S^{\Ker(\lambda_X)}, S^{\Cok(\lambda_X)}]^{\GG}
  \rightarrow\bigoplus_j\H^0(WH;
                  [S^{\Ker(\lambda_X)^{H(j)}},
                    S^{\Cok(\lambda_X)^{H(j)}}]^{\TT})
\end{displaymath} 
for each subgroup~$H$ of~$G$. If~$H_1$ and~$H_2$ are conjugate
subgroups of~$G$, their Weyl groups are isomorphic and there is a
diffeomorphism between~$X/H_1$ and~$X/H_2$ which respects the actions
of the Weyl groups. Therefore, it is not necessary to consider all
subgroups~$H$ of~$G$, but only a set of representatives of the
conjugacy classes~$(H)$. The product
\begin{equation}\label{old ghost map}
  [S^{\Ker(\lambda_X)}, S^{\Cok(\lambda_X)}]^{\GG}
  \!\!\rightarrow\!\!\bigoplus_{(H),j}\!\H^0(WH;
                  [S^{\Ker(\lambda_X)^{H(j)}},
                    S^{\Cok(\lambda_X)^{H(j)}}]^{\TT})
\end{equation}
of the above maps over such a set will be referred to as the {\it
  ghost map}. As is apparent from the interpretation given for the
terms involved, it codifies the relationship between the equivariant
and the ordinary invariants. Therefore, one would like to understand
it as well as possible. This is the aim of the rest of this text.

\subsection{The ghost map -- second version}

The following proposition gives a translation from the ghost map
(\ref{old ghost map}) to something which has smaller groups of equivariance:
if~$G$ is finite, they will be finite, too.

\begin{prop}\label{prop:finite version of ghost map}
  Assume that the positivity conditions {\upshape
    (\ref{eq:first_condition})} and {\upshape
    (\ref{eq:second_condition})} are fulfilled, such
  that~$[S^{\Ker(\lambda_X)}, S^{\Cok(\lambda_X)}]^{\GG}\cong[\Cset
  P(V)_+,S^{W-1}]^G$ as in Proposition~\ref{reduction}. Then the ghost
  map~{\upshape (\ref{old ghost map})} is isomorphic to a map
  \begin{equation}\label{new ghost map}
    [\Cset P(V)_+,S^{W-1}]^G\longrightarrow\bigoplus_{(H)}\H^0(WH;[\Cset
    P(V)^H_+,S^{W^H-1}]).
  \end{equation}
\end{prop}

The map~(\ref{new ghost map}) is built similar to the ghost map:
in each factor, first restrict to the fixed points, then forget the
equivariance.

The proof of~Proposition~\ref{prop:finite version of ghost map} starts
with a remark on the indexing. At first sight, it seems that the
indexing over~$j$ has disappeared in~(\ref{new ghost map}), but the
components of the fixed point set~$\Cset P(V)^H_+$ are indexed
by~$J_H(\sigma_X)$:

\begin{prop}\label{prop:fixed_points_of_CP}
  Let~$H$ be a subgroup of~$G$. Then
  \begin{displaymath} 
    \Cset P(V)^H=\coprod_{j \in J_H(\sigma_X)}\Cset P(V^{H(j)}).
  \end{displaymath} 
  If~$J_H(\sigma_X)$ is empty, this means that~$\Cset P(V)^H$ is empty
  as well.
\end{prop}

Note that the summands~$\Cset P(V^{H(j)})$ are invariant under the
$WH$-action. Therefore, the proposition gives a decomposition
of~$\Cset P(V)^H$ as a~$WH$-space, and enables us to prove
Proposition~\ref{prop:finite version of ghost map}.

\begin{proof}
  As a consequence of Proposition~\ref{prop:fixed_points_of_CP}, and
  since~$H(j)$ acts on~$W$ via~$H$, we have~$W^H=W^{H(j)}$, and
  consequently
  \begin{displaymath}
  [\Cset P(V)^H_+,S^{W^H-1}]^{WH} = \bigoplus_j\,[\Cset
    P(V^{H(j)})_+,S^{W^{H(j)}-1}]^{WH}.
  \end{displaymath}
  Hence, the composition of this identification with
  the sum over~$j$ of the isomorphisms
  \begin{displaymath}
    [\Cset P(V^{H(j)})_+,S^{W^{H(j)}-1}]^{WH} \stackrel{\cong}\longrightarrow
    [S^{V^{H(j)}},S^{W^{H(j)}}]^{\WW H(j)}
  \end{displaymath}
  from Proposition~\ref{reduction} can be used to define the dashed
  isomorphism on the right hand side of the diagram
  \begin{center}
    \mbox{
      \xymatrix@C=80pt{
        [S^V,S^W]^{\GG} \ar[r]^-{\bigoplus_j(?)^{H(j)}} &
        \bigoplus_j[S^{V^{H(j)}},S^{W^{H(j)}}]^{\WW H(j)}
        \\
          [\Cset P(V)_+,S^{W-1}]^G \ar[r]^-{(?)^H}\ar[u]^-{\cong} & 
          [\Cset P(V)^H_+,S^{W^H-1}]^{WH}\ar@{-->}[u]_-{\cong}
    }}
  \end{center}
  so that this diagram commutes. Also, using the dashed arrow on the left hand
  side and another instance of the isomorphism from Proposition~\ref{reduction}
  for the right hand side, the diagram
  \begin{center}
    \mbox{\xymatrix{\bigoplus_j[S^{V^{H(j)}},S^{W^{H(j)}}]^{\WW
          H(j)} \ar[r] &
        \bigoplus_j\H^0(WH;[S^{V^{H(j)}},S^{W^{H(j)}}]^{\TT})
        \\
        [\Cset P(V)^H_+,S^{W^H-1}]^{WH}\ar@{-->}[u]^-{\cong} \ar[r] &
        \H^0(WH;[\Cset P(V)^H_+,S^{W^H-1}])\ar[u]_-{\cong} }}
  \end{center}
  commutes. By placing the two preceding diagrams next to each other
  and summing
  over the~$(H)$ one obtains the result. 
\end{proof}

The rest of this text will be concerned with the ghost map in the
form~(\ref{new ghost map}).

\subsection{Localisation}

Now that the main case of interest has been reduced to equivariant
stable homotopy theory with a finite group of equivariance, one can
use the fact that this theory is easier to understand as soon as the
order of the group is inverted.

\begin{thm}\label{thm:hkr}
  Under the assumptions of Proposition~\ref{prop:finite version of
    ghost map}, the ghost maps~{\upshape (\ref{old ghost map})}
  and~{\upshape (\ref{new ghost map})}, which compare the equivariant
  invariant with the family of the ordinary invariants, become
  isomorphisms after inverting the order of the group: kernel and
  cokernel are finite abelian groups whose order is some power of the
  order of the group.
\end{thm}

\begin{proof}
  This is a consequence of the following general result: Let~$G$ be a
  finite group,~$M$ and~$N$ two pointed~$G$-CW-complexes,~$M$ finite.
  Consider the \textit{localisation map}
  \begin{equation}\label{hkr}
    [M,N]^G\longrightarrow\bigoplus_{(H)}\H^0(WH;[M^H,N^H])
  \end{equation}
  which is constructed as the ghost map by first passing to fixed
  points and then taking invariants. This is an isomorphism after
  inverting the order of~$G$. (See for example Lemma~3.6 on page 567
  of \cite{HKR}. The reader might also like to have a look at
  \cite{Greenlees:GASS}.) The ghost map is the localisation
  map~(\ref{hkr}) for~\hbox{$M=\Cset P(V)_+$} and~\hbox{$N=S^{W-1}$}.
  The final comment follows since the groups in question are finitely
  generated.
\end{proof}

In other words, away from the group order, the~$G$-equivariant
Bauer--Furuta invariant of~$(X,\sigma_X)$ contains the same information
as all of the ordinary invariants of all of its
quotients~$(X/H,\sigma_{X/H})$ for subgroups~$H$ of~$G$ and complex
spin structures~$\sigma_{X/H}$ on~$X/H$ which pull back to~$\sigma_X$.
In particular, the situation is understood rationally. The following
sections will address the integral and torsion information.

%%%%%%%%%%%%%%%%%%%%%%%%%%%%%%%%%%%%%%%%%%%%%%%%%%%%%%%%%%%%%%%%%%

\section{\bf Indices and groups}

Recall from Theorem~\ref{thm:existence} that if~$X$ is a 4-manifold
with an action of a compact Lie group~$G$, and if~$\sigma_X$ is
a~$G$-invariant complex spin structure on~$X$, there is an equivariant
Bauer--Furuta invariant which lives in the
group~\hbox{$[S^{\Ker(\lambda)},S^{\Cok(\lambda)}]^{\GG}$}. Up to
isomorphism, this group depends only on the
class~$[\Ker(\lambda)]-[\Cok(\lambda)]$ in~$\RO(\GG)$. This section
contains a few remarks on these groups, starting with the case of the
ordinary Bauer--Furuta invariants, where~$G=1$ and~$\GG=\TT$, and
finishing with the case where~$G$ is a finite group of Galois
symmetries.

\subsection{The ordinary situation}

The linearisation of the monopole map decomposes as~a sum of the
complex Dirac operator, which sends a spinor~$\phi$ to the
spinor~$D_A(\phi),$ and the real map
\begin{equation}\label{Atiyah complex}
  \Omega^1(L\TT) \rightarrow
  \Omega^+(L\TT)\oplus\bar\Omega^0(L\TT),
\end{equation}
which maps~$a$ to~$(\d^+a,[\d^*a])$. The complex index of the Dirac
operator~is
\begin{displaymath}
  a(\sigma_X)=\frac{c_1^2(\sigma_X)-s}{8}.
\end{displaymath} 
Since the first Betti number is assumed to vanish, the real index of
the operator~(\ref{Atiyah complex}) is~$-b^+$. If we interpret the
index of the linearisation as a~$\TT$-representa\-tion, the circle
acting trivially on real vector spaces and by scalar multiplication on
complex vector spaces, then the $\TT$-index of the linearisation is the
class~\hbox{$a(\sigma_X)[\Cset]-b[\Rset]$} in $\RO(\TT)$, with~$a(\sigma_X)$
as above and~$b=b^+$.  The real virtual dimension of the~$\TT$-index
is given by~$2a(\sigma_X)-b$, which equals
\begin{equation}\label{realindex}
  \frac{c_1^2(\sigma_X)-(2e+3s)}{4}+1.
\end{equation}
The number
\begin{displaymath}
  d(\sigma_X)=\frac{c_1^2(\sigma_X)-(2e+3s)}{4}
\end{displaymath} 
will be called the \emph{degree} of the complex spin structure. This
number vanishes if and only if the complex spin structure comes from
an almost complex structure. From~a different point of view, the
number~$d(\sigma_X)$ is the virtual dimension of the moduli space of
solutions to the Seiberg-Witten equations. The virtual dimension
(\ref{realindex}) of the index is $d(\sigma_X)+1$, since the monopole
map describes the~$\TT$-equivariant situation, which corresponds to
a~$\TT$-space over the moduli space.

The following discussion assumes that~$b\geqslant 2$.
If~$a\leqslant0$, one may reorder to show that the
group~$[S^{\Ker(\lambda)},S^{\Cok(\lambda)}]^\TT$ in question is
isomorphic to the group~$[S^0,S^{b\Rset-a\Cset}]^{\TT}$, where
\hbox{$b\Rset-a\Cset$} is a~$\TT$-representation with~$b$-dimensional
trivial summand. It follows that this group vanishes. From now on, one
can focus on the case where~$a$ is positive. Then one reorders the
index of~$\lambda$ according to real and complex parts and sees that
the group in question is isomorphic to the
group~$[S^{a\Cset},S^{b\Rset}]^{\TT}$.  If~$a\geqslant0$
and~$b\geqslant 2$, there is an isomorphism
\begin{equation}\label{reduction mod TT}
  [S^{a\Cset},S^{b\Rset}]^{\TT} \cong[\Cset P^{a-1}_+,S^{b-1}]
\end{equation}
of groups. (This is Proposition~\ref{reduction} in the case when~$G$
is trivial.) Therefore, one needs to know the structure of the
groups~$[\Cset P^{a-1}_+,S^{b-1}]$. Let us regard~$a$ as fixed and
sort these groups by degree~$d$, using the
relation~\hbox{$b-1=2(a-1)-d$}, so that the groups are~$[\Cset
P^{a-1}_+,S^{2(a-1)-d}]$. These groups are finitely generated and
vanish for negative~$d$. They vanish rationally, except for even
integers~$d$ that satisfy \hbox{$0\leqslant d\leqslant 2(a-1)$}. In
those cases, they have rank~$1$. Apart from divisibility properties of
the Hurewicz image, about which nothing will be said here, the most
interesting information is the structure of the torsion. If~$\ell$
is~a prime, there is no~$\ell$-power torsion in~$[\Cset
P^{a-1}_+,S^{2(a-1)-d}]$ for~$d<2\ell-3$. If~\hbox{$d=2\ell-3$},
the~$\ell$-power torsion in~$[\Cset P^{a-1}_+,S^{2(a-1)-d}]$ is a
group of order~$\ell$ if~$a$ is~a multiple of~$\ell$ and is trivial
else.

\subsection{The equivariant situation}

As a starting point for later computations, it will be useful to have
a more concrete description of the groups appearing in the ghost
map~(\ref{new ghost map}). It will be assumed
that~\hbox{$b=b^+(X/G)\geqslant2$}. This implies $b^+(X/H)\geqslant2$
for all subgroups~$H$ of $G$.

In general, if a group~$G$ acts on the 4-manifold~$X$, one may consider the
equivariant Euler characteristic
\begin{displaymath}
  e_G(X)=[\H^0(X)]-[\H^1(X)]+[\H^2(X)]-[\H^3(X)]+[\H^4(X)]
\end{displaymath} 
and the equivariant signature
\begin{displaymath}
  s_G(X)=[\H^+(X)] -[\H^-(X)]
\end{displaymath} 
of~$X$, which are elements in the representation ring~$\RO(G)$. In the
special case when the group~$G$ is finite and acts freely on~$X$,
these are simply given by~\hbox{$e_G(X)=e(X/G)\cdot[\Rset G]$} and~\hbox{$s_G(X)=s(X/G)\cdot[\Rset G]$}, as follows for example from the
Atiyah-Bott-Lefschetz fixed point formula. Rather than only
in~$e_G(X)$ and~$s_G(X)$, one might also be interested in
the~$G$-subrepresentations~$\H^+(X)$ and~$\H^-(X)$ of~$\H^2(X)$. The
assumption~$b^1=0$ implies that we have the equation~\hbox{$e_G(X)=2[\Rset]+[\H^+(X)]+[\H^-(X)]$}
holds in~$\RO(G)$. It follows that
\begin{displaymath}
  [\H^{\pm}(X)]=\frac{e_G(X)\pm s_G(X)}{2}-[\Rset].
\end{displaymath}
The~$G$-representation~$W=\H^+(X)$ is isomorphic to~$(b+1)\Rset G-1$.
Therefore, one has~\hbox{$S^{W-1} \cong S^{(b+1)\Rset G-2}$}. Starting
from this, one can easily sort out the~$H$-fixed points for the
various subgroups~$H$ of~$G$.

Let now~$\sigma_X$ be a~$G$-invariant complex spin structure on~$X$
such that an actual~$\GG$-representation~$V$ represents
the~$\GG$-index of the Dirac operator. It is not as easy to describe
the~$G$-space~$\Cset P(V)$, since the~$G$-action need not be induced
by a~$G$-action on~$V$. But, more generally, if the restriction
of~$\Cset P(V)$ from~$G$ to a subgroup~$H$ is not induced by
an~$H$-action on~$V$, the fixed point set~$\Cset P(V)^H$ is empty, see
Proposition~\ref{prop:fixed_points_of_CP}. Therefore, these~$H$ do not
contribute to the target of the ghost map. One may therefore assume
that the restriction of~$\Cset P(V)$ from~$G$ to~$H$ is induced by
an~$H$-action on~$V$. The choices of the~$H$-actions are indexed by
the elements~$j$ in~$J_H(\sigma_X)$. The
complex~$H(j)$-representation~$V$ represents the~$H(j)$-equivariant
index of the Dirac operator~$D_X$. The usual indices~$\ind(D_X)$
and~$\ind(D_Y)$ are integers. Since~$D_X$ is~$H(j)$-equivariant, the
equivariant index $\ind_{H(j)}(D_X)$, which by definition
is~\hbox{$[\Ker(D_X)]-[\Cok(D_X)]$}, lives in the complex
representation ring~$\RU(H(j))$. One may deduce that the
equality~$\ind_{H(j)}(D_X)=\ind(D_Y)\cdot[\Cset H(j)]$ holds
in~\hbox{$\RU(H(j))$} as above. Thus,~$V$ is a multiple of the
regular~$H(j)$-representation~$\Cset H(j)$. The multiplicity is the
index~$a(\sigma_{X/H}(j))$ of the complex spin structure on the
quotient~$X/H$ which pulls back to~$\sigma_X$. While this description
seems to depend on~$j$, the~$H$-space~$\Cset P(V)$ does not. Also, the
indices~$a(\sigma_{X/H}(j))$ are independent of~$j$ so that one may
write~$a(\sigma_{X/H})$ unambiguously. The~$H$-fixed point set~$\Cset
P(V)^H$ consists of a disjoint union -- indexed by the
set~$\Hom(H,\TT)$~-- of complex projective spaces of complex
dimension~$a(\sigma_{X/H})-1$, see
Proposition~\ref{prop:fixed_points_of_CP} again.

For use in the following two sections,~I will make explicit what the
previous remarks mean if the order of~$G$ is a prime~$p$.

\subsection{Groups of prime order}

In the case when the order of the group~$G$ is a prime~$p$,
any~$G$-invariant complex spin structure~$\sigma_X$ on~$X$ can be
made~$G$-equivariant. There are~$p$ complex spin structures on the
quotient~$X/G$ which pull back to~$\sigma_X$. Let~$a=a(\sigma_{X/G})$
and~$d=d(\sigma_{X/G})$ be the index and the degree, respectively, of
the corresponding complex spin structures on the quotient~$X/G$.
Recall that~\hbox{$2a-d=b+1$} for~\hbox{$b=b^+(X/G)\geqslant2$}.

\begin{prop}\label{prop:groups for order p} 
  If the order of the group $G$ is a prime number $p$, up to
  isomorphism, the source of the ghost map is
  \begin{displaymath}
    [\Cset P(a\Cset G)_+,S^{(2a-d)\Rset G-2\Rset}]^G.
  \end{displaymath}
  The target is
  \begin{displaymath}
    \H^0(G;[\Cset P^{ap-1}_+,S^{(2a-d)p-2}])\oplus [\Cset
      P^{a-1}_+,S^{2a-d-2}]^{\oplus p},
  \end{displaymath}
  up to isomorphism.
\end{prop}

In the statement, the~$G$-action on~$[\Cset P^{ap-1}_+,S^{(2a-d)p-2}]$
comes from the identification of that group with~$[\Cset P(a\Cset
G)_+,S^{(2a-d)\Rset G-2\Rset}]$. 

This might be a good point to insert a comment on the structure
of~$G$-modules such as~$[M,N]$ for finite $G$-CW-complexes $M$ and
$N$. The~$G$-action on a group like that is trivial if the group~$G$
acts on~$M$ and~$N$ via homotopically trivial maps. The latter will
always be the case for complex projective spaces~\hbox{$M=\Cset
  P(V)_+$} with a linear action. For spheres~$N=S^W$ it is the case if
and only if~$W$ is orientable, i.e. if~$\det(W)$ is trivial. In
particular, for odd order groups the action is always trivial.

The ghost map in the two cases~$d=0$ and~$d=1$ will be discussed in
more detail in the following two sections. As both phenomena -- a
non-trivial cokernel and a non-trivial kernel -- already occur in
these two examples, it does not seem to be illuminating to proceed and
discuss the cases where~\hbox{$d\geqslant 2$}.

%%%%%%%%%%%%%%%%%%%%%%%%%%%%%%%%%%%%%%%%%%%%%%%%%%%%%%%%%%%%%%%%%%

\section{\bf Degree zero}

In this section, some applications of calculations in equivariant
stable homotopy theory to the Bauer--Furuta invariants will be
described in the case when all complex spin structures involved have
degree zero. Since in this case the ordinary Bauer--Furuta invariants
can be identified with the Seiberg-Witten invariants, the results of
this section apply to those as well. As examples, some elliptic
surfaces will be discussed. This will lead to congruences between
certain binomial co-efficients.

\subsection{General results}

Let~$G$ be a finite group of prime order~$p$ which may be even or odd.
Let~$X$ be a~4-manifold with a~$G$-action which preserves a complex
spin structure~$\sigma_X$ on~$X$. In the previous section, a ghost map
has been assembled which sends~$m_G(X,\sigma_X)$ to~$m(X,\sigma_X)$
and the family of the invariants~\hbox{$m(X/G,\sigma_{X/G}(j))$}
for $j$ in~$J_G(\sigma_X)$. Since one has~\hbox{$d(\sigma_X)=p\cdot
  d(\sigma_{X/G}(j))$}, the degree zero case is the case where the
degrees of all the complex spin structures involved are zero. 

Assume as before that~\hbox{$b^+(X/G)\geqslant 2$} holds.
This implies that~\hbox{$a(\sigma_{X/G})\geqslant 2$}. In particular, the
hypotheses of Proposition~\ref{prop:finite version of ghost map} are
fulfilled, so that the ghost map takes the form~(\ref{new ghost map}).
Proposition~\ref{prop:groups for order p} then implies that the target
of the ghost map is isomorphic to a free abelian group of rank~$p+1$:
As regards the first summand, note that~\hbox{$[\Cset
  P^{ap-1}_+,S^{2ap-2}]\cong\Zset$}, and the action of~$G$ on it is
trivial. This is clear for odd~$p$. If~$p=2$ it follows from the fact
that~\hbox{$W\cong a\Cset G-\Cset$} has the structure of a complex
representation, so that the action on the sphere~$S^W$ preserves a
chosen orientation. For the other~$p$ summands, note that there is an isomorphism~$[\Cset
P^{a-1}_+,S^{2a-d-2}]\cong\Zset$.

There is an equivariant stable Hopf theorem, see~\cite{Szymik:Hopf},
which tells us that in this situation the ghost map is injective with
a cokernel of order~$p$, and that the elements in the image are
characterised by a certain congruence. In our situation this reads as
follows.

\begin{thm}\label{thm:relations}
  Assume that~$b^+(X/G)\geqslant2$. In the degree zero case, the
  equivariant invariant~$m_G(X,\sigma_X)$ is determined by the
  ordinary invariants~$m(X,\sigma_X)$ and~$m(X,\sigma_{X/G}(j))$, where~$j$
  ranges over~$J_G(\sigma_X)$. The relation
  \begin{displaymath} 
    m(X,\sigma_X) \equiv \sum_{j \in J_G(\sigma_X)}
    m(X/G,\sigma_{X/G}(j))\;\mod\;p
  \end{displaymath} 
  is satisfied by the latter.
\end{thm}

As has been shown in~\cite{BauerFuruta}, in the degree zero case, the
Bauer--Furuta invariants can be identified with the integer valued
Seiberg-Witten invariants. Therefore, the latter must satisfy the same
relations. See~\cite{RuanWang} for the case of
involutions. Theorem~\ref{thm:relations} will now be illustrated with
elliptic surfaces, a class of examples where the Seiberg-Witten
invariants are known.
 
\subsection{Elliptic surfaces} 

Relatively minimal regular elliptic surfaces with at most two multiple
fibres will be considered. (For background on elliptic surfaces see the
articles~\cite{Dolgachev} and~\cite{Ue} or the
books~\cite{FriedmanMorgan:Book},~\cite{Friedman:Book}
and~\cite{GompfStipsicz}.) Let~$m_1$ and~$m_2$ denote the
multiplicities of the two fibres~$F^1$ and~$F^2$, respectively.
Setting~$p=\gcd(m_1,m_2)$, the fundamental group of the surface is
cyclic of order~$p$. In particular~$b^1=0$. The other invariants can
be computed from the geometric genus~$p_g$ as follows.  The Euler
characteristic is~\hbox{$e=12(p_g+1)$}. It follows
that~\hbox{$b^2=12p_g+10$}. In fact~\hbox{$b^+=2p_g+1$}
and~\hbox{$b^-=10p_g+9$}.  The signature is given
by~\hbox{$s=-8(p_g+1)$}. The assumption~\hbox{$b^+\geqslant2$} translates
into~$p_g\geqslant1$. In particular, Dolgachev surfaces are not
allowed since they have~$p_g=0$.

Let me abuse (additively written) divisors to denote the corresponding
line bundles, the canonical divisor~$K$ corresponding
to~$\Lambda^2T^*$. For complex surfaces the isomorphism classes of
complex spin structures are canonically para\-metri\-sed by the
isomorphism classes of line bundles, the trivial line bundle
corresponding to a complex spin structure
with~\hbox{$W^+=\Cset\oplus\Lambda^2T$} and~\hbox{$W^-=T$}, having
determinant line bundle $-K$. It follows that~$p_g+1$ is the index of
the Dirac operator for the canonical complex spin structure.

The Seiberg-Witten invariants of the elliptic surfaces have been
computed, see for example \cite{Brussee},
\cite{FriedmanMorgan:BasicClasses} and
\cite{FriedmanMorgan:Multiplicities}. In order to describe the result,
some notation needs to be introduced. For an integer~\hbox{$i\geqslant0$} let~$[i]=\{0,\dots,i\}$. Write~$[i_1,i_2,i_3]$
for~$[i_1]\times[i_2]\times[i_3]$. For~$(a,b,c)$
in~\hbox{$[p_g-1,m_1-1,m_2-1]$}, there is the effective divisor
\begin{displaymath} 
  D(a,b,c)=aF+bF^1+cF^2 
\end{displaymath} 
on the elliptic surface. At the two extremes,~$D(p_g-1,m_1-1,m_2-1)$
is the canonical divisor~$K$, and~$D(0,0,0)$ is the trivial divisor.
Later,~$D(a)$ will be written instead of $D(a,0,0)$. The line bundles
leading to non-trivial Seiberg-Witten invariants are of the
form~\hbox{$K-2D(a,b,c)$} for triples~$(a,b,c)$ in the
cube~\hbox{$[p_g-1,m_1-1,m_2-1]$}. The value of the invariant for the
corresponding complex spin structure is~\hbox{$(-1)^a\binom{p_g-1}{a}$},
independent of~$b$ and~$c$.

Now let us consider a Galois covering situation. The notation~$E(n)$
is used for an elliptic surface with~$p_g=n-1$ and no multiple fibres.
Multiplicities will appear as indices:~$E(n)_{m_1,m_2}$. Note
that~$n\geqslant2$ by our assumption on~$p_g$. The shift from~$p_g$
to~$n$ is justified by the fact that~$n=p_g+1$ behaves well under
coverings. In fact, if~$p$ divides~$m_1$ and~$m_2$, there is a Galois
covering
\begin{displaymath} 
  E(pn)_{m_1/p,m_2/p} \longrightarrow E(n)_{m_1,m_2} 
\end{displaymath} 
with Galois group of order~$p$. If~$p=\gcd(m_1,m_2)$, this is a
universal covering. For example, the universal covering of an Enriques
surface~$E(1)_{2,2}$ is a~K3-surface~$E(2)$; however, the condition~$n
\geqslant 2$ is not satisfied by the Enriques surfaces.  As the
simplest examples,~I would like to discuss the Galois coverings~\hbox{$E(pn)
\rightarrow E(n)_{p,p}$} for~$n\geqslant2$ and any prime number~$p$.

Now the Seiberg-Witten invariants of the covering surface~\hbox{$X=E(pn)$}
take the value~$(-1)^d\binom{pn-2}{d}$ on the
classes~\hbox{$K_X-2D_X(d)$} for~$d$ in~\hbox{$[pn-2]$}. On the
covered surface~$Y=E(n)_{p,p}$ they are given by the
number~$(-1)^a\binom{n-2}{a}$ on the classes~\hbox{$K_Y-2D_Y(a,b,c)$} for a
triple~$(a,b,c)$ in the cube~\hbox{$[n-2,p-1,p-1]$}. Since the
canonical divisor~$K_Y$ pulls back to~$K_X$, the class~$D_Y(a,b,c)$
pulls back to~\hbox{$D_X(pa+b+c)$}. The relations from Theorem~\ref{thm:relations} are now equivalent to a congruence between binomial co-efficients. In order to make these explicit,  let again~$p$ be a prime number, and~$n$ be an integer,~\hbox{$n\geqslant2$}. Then, for any integer~$d$ such that~$0\leqslant d\leqslant pn-2$, the relations are equivalent to the congruence
\begin{equation}\label{elliptic_congrunce} 
    (-1)^d\binom{pn-2}{d} \equiv \sum_{(a,b,c)} 
    (-1)^a\binom{n-2}{a}\;\mod\;p,
\end{equation}  
where the sum ranges over the triples~$(a,b,c)$ in~\hbox{$[n-2,p-1,p-1]$} which satisfy the relation~$pa+b+c=d$. Note that the terms on the right hand side of
(\ref{elliptic_congrunce}) do not depend on~$b$ and~$c$; these enter
only in the summation set. Also, the sets summed over do not always
have~$p$ elements. That reflects the fact that for some of the
pre-images of~$K_X-2D_X(d)$ the Seiberg-Witten invariant vanishes. For
example, this is the case for~$d=0$ and~\hbox{$d=pn-2$}, in other words for
the canonical and the anti-canonical complex spin structures. The reader is encouraged to find her or his own elementary proof of~(\ref{elliptic_congrunce}) so as to verify this instance of Theorem~\ref{thm:relations} by hand.

%%%%%%%%%%%%%%%%%%%%%%%%%%%%%%%%%%%%%%%%%%%%%%%%%%%%%%%%%%%%%%%%%%

\section{\bf Degree one}

As in the previous section, let~$G$ be a finite group whose order is a
prime~$p$. Again, calculations in equivariant stable homotopy theory
will be applied to study Bauer--Furuta invariants of Galois
coverings~$X\rightarrow X/G$. This time, however, the complex spin
structures on the quotient~$X/G$ will have degree one. 

\subsection{General results}

If the complex spin structures on the quotient~$X/G$ have degree one,
those on~$X$ consequently have degree~$p$. We will show that the class
of the equivariant invariant~$m_G(X,\sigma_X)$ is in general not
determined by the classes of the ordinary invariants

By Proposition~\ref{prop:groups for order p}, the target of the ghost
map is isomorphic to
\begin{displaymath}
  \H^0(G;[\Cset P(a\Cset G)_+,S^{(2a-1)\Rset G-2\Rset}])\oplus [\Cset
    P^{a-1}_+,S^{2a-3}]^{\oplus p}.
\end{displaymath}
If~$\ell$ is a prime number, the~$\ell$-power torsion of~$[\Cset
P^{a-1}_+,S^{2a-3}]$ is trivial except maybe for~$\ell=2$.
The~$\ell$-power torsion of
\begin{displaymath}
  [\Cset P(a\Cset G)_+,S^{(2a-1)\Rset G-2\Rset}]=[\Cset
    P^{ap-1}_+,S^{(2a-1)p-2}]
\end{displaymath} 
is trivial except maybe if~$2\ell-3\leqslant p$. If~$p\geqslant 5$,
this means that only~$\ell$-power torsion for~$\ell<p$ appears in the
target of the ghost map. This is the main case. Let me briefly comment
on the other two primes~$p=2$ and~$p=3$ before~I return to it in more
detail.

In the case~$p=2$, the~$2$-torsion in~$[\Cset P^{a-1}_+,S^{2a-3}]$
becomes relevant, and the~$G$-action on~$[\Cset P(a\Cset
G)_+,S^{(2a-1)\Rset G-2\Rset}]$ will have to be discussed. The latter
group sits in an extension
\begin{displaymath}
  0\longrightarrow\Zset/2\longrightarrow[\Cset P(a\Cset G)_+,S^{(2a-1)\Rset
  G-2\Rset}]\longrightarrow\Zset\longrightarrow0
\end{displaymath} 
which comes from the Hurewicz map. This time, however, the action
of the group~$G$ on the sphere~\hbox{$S^{(2a-1)\Rset G-2\Rset}$} is not
orientation preserving, so the action on~$\Zset$ is non-trivial.  As a
consequence, the target of the ghost map is isomorphic to~$\Zset/2$
for odd~$a$ and to~$(\Zset/2)^{\oplus 3}$ for even $a$. If~$p=3$, the
target of the ghost map is isomorphic to one copy of~$\Zset/3$ and
maybe some~$2$-torsion.

Now back to the main case. As mentioned above, if~$p\geqslant 5$, the
target of the ghost map is torsion of order away from the group
order~$p$. In particular, by Theorem~\ref{thm:hkr}, the ghost map is
surjective in this case. The kernel is equal to the~$p$-power torsion
in the source~$[\Cset P(a\Cset G)_+,S^{(2a-1)\Rset G-2\Rset}]^G$ of
the ghost map. Calculations with an Adams spectral sequence show,
see~\cite{Szymik:Diss} or~\cite{Szymik:Periodicity}, that the ghost
map is not injective in this case. This means for the Bauer--Furuta
invariants that there are several elements which the ghost map sends
to the collection of the ordinary invariants.

\begin{thm}\label{thm:not injective}
  The homotopy classes of all the ordinary
  invariants~$m(X,\sigma_X)$ and~$m(X/G,\sigma_{X/G}(j))$ for all~$j$
  in~$J_G(\sigma_X)$ do not determine the homotopy class of the
  equivariant invariant~$m_G(X,\sigma_X)$ in general.
\end{thm}

This raises the question of whether there are 4-manifolds
with~$G$-action reali\-sing the different possibilities opened up by
homotopy theory or not: Do there \hbox{exist}~$(X,\sigma_X)$
and~$(X',\sigma_{X'})$ with free~$G$-actions which have different
equivariant invariants but which have the same ordinary invariants? In
particular, if there is a pair~$(X,\sigma_X)$ with free~$G$-action for
which the ordinary in\-variants~$m(X,\sigma_X)$
and~$m(X/G,\sigma_{X/G}(j))$ are zero for all~$j$, does it follow that
the equivariant invariant~$m_G(X,\sigma_X)$ is zero as well? 
A natural place to look for examples of complex spin structures of degree one is among connected sums of those of degree zero; but, as will be shown in the following final subsection, in that case the invariants will all be determined by non-equivariant data. As at the time of writing there do not seem to be any other relevant examples known which do not fit into this pattern, new geometric constructions seem to be called for to settle this question; Theorem~\ref{thm:not injective} marks the scope of homotopy theory here.

\subsection{Connected sums}

Recall from~\cite{Bauer:Sum} the following. Let~$X_1$ and~$X_2$ be two
4-manifolds with complex spin structures~$\sigma_{X_1}$
and~$\sigma_{X_2}$. Then there is a canonical complex spin
structure~$\sigma_{X_1\#X_2}$ on the connected sum~$X_1\#X_2$ which
restricts to the given one on each summand. The connected sum theorem
says that~$m(X_1\#X_2,\sigma_{X_1\#X_2})$ can be identified with the
smash product~\hbox{$m(X_1,\sigma_{X_1}) \wedge m(X_2,\sigma_{X_2})$}.
More generally, let~$X^{\pm}$ be two 4-manifolds, oriented and
compact, but not necessarily connected. The boundaries~$\partial
X^{\pm}$ are to be identified with a collection~$[-L,L]\times
S^3\times\Lambda$ of necks of length~$2L$.  Here~$\Lambda$ is a finite
index set. Given a permutation~$\tau$ of~$\Lambda$, one may build an
oriented closed 4-manifold~$X^-\cup_{\tau}X^+$ by gluing as indicated
by~$\tau$. This carries a complex spin structure if the~$X^{\pm}$ do
so in a way compatible with the identification over the necks. Then,
if~$\tau_1$ and~$\tau_2$ are even permutations and the first Betti
numbers of~$X^-\cup_{\tau_1}X^+$ and~$X^-\cup_{\tau_2}X^+$ are zero,
there is an~(explicitly described) identification
\begin{equation}\label{identification}
  m(X^-\cup_{\tau_1}X^+) \cong m(X^-\cup_{\tau_2}X^+),
\end{equation}
omitting the evident complex spin structures from the
notation. See~\cite{Bauer:Sum} again.

In order to describe an equivariant extension of this result, let us
assume that~$G$ acts freely on~$X^{\pm}$. (For the matter of this
paragraph,~$G$ can be any finite group.) We will also assume that~$G$
acts on the complex spin structures in a compatible way, so that~$\GG$
is identified with~\hbox{$\TT\times G$} for all components. This time
the boundaries~$\partial X^{\pm}$ are required to be identified with a
collection~\hbox{$G\times [-L,L]\times S^3\times\Lambda$} of necks.
As above, given a permutation~$\tau$ of~$\Lambda$, one may build an
oriented closed~4-manifold~$X^-\cup_{\tau}X^+$ by gluing. This carries
a free $G$-action. Note that an even permutation of~$\Lambda$ induces
an even permutation of~$G\times\Lambda$. Therefore, if the condition
on the first Betti numbers is satisfied, there is an
identification~(\ref{identification}). It is easily checked that the
identification maps and homotopies used in \cite{Bauer:Sum}
are~$G$-equivariant. This implies that also the equivariant
invariants~$m_G(X^-\cup_{\tau_1}X^+)$ and~\hbox{$m_G(X^-\cup_{\tau_2}X^+)$} can be identified. As in the
non-equivariant setting, this leads to results on connected sums, as
will now be exemplified.

Let~$X\rightarrow Y$ be a Galois~$G$-covering. If a complex spin
structure~$\sigma_Y$ on~$Y$ with degree zero is given, the
pullback~$\sigma_X$ on~$X$ has degree zero as well. Again, the
notation $\sigma_Y(j)$ will be used for the different complex spin
structures on~$Y$ which pull back to~$\sigma_X$ on~$X$. Let us choose
an additional 4-manifold~$Z$ and a complex spin structure~$\sigma_Z$
of degree zero. To be on the safe side, let us also assume
that~$b^+\geqslant 2$. (For example,~$Z$ may be taken to be a
K3-surface and~$\sigma_Z$ the canonical spin structure.) For the rest
of this section let us work with these chosen complex spin structures
and their connected sums. If confusion is unlikely, they can be
suppressed from the notation. In the same vein, to improve legibility,
let us write~$Y(j)$ for~$Y$ with the complex spin
structure~$\sigma_Y(j)$ and similarly~$Y(j)\#Z$ to indicate the
relevant complex spin structures on~$Y\#Z$. If~$S^3$ denotes the
separating~3-sphere in~$Y\#Z$, there is an equivariant connected
sum~$X\#_G(G\times Z)$ along~$G\times S^3$, and one may consider
the~$G$-coverings
\begin{equation}\label{sum_covering}
  X\#_G(G\times Z)\longrightarrow Y(j)\#Z.
\end{equation}
Using the connected sum theorem of Bauer, which
identifies~$m(Y(j)\#Z)$ with the smash product~\hbox{$m(Y(j))\wedge
  m(Z)$}, we see that the ordinary invariants of each of the~$Y(j)\#
Z$ can be described in terms which shall be assumed to be known,
namely the ordinary invariants of the~$Y(j)$ and~$Z$. That theorem
also identifies the ordinary invariant~$m(X\#_G(G\times Z))$ with the
smash product~$m(X) \wedge m(Z)^{\wedge p}$. This is zero as soon
as~$p\geqslant5$.  (If~$a\geqslant2$ then every~$m$
in~$[S^{a\Cset},S^{2a-1}]^{\TT}$ satisfies~$m^{\wedge5}=0$.) To sum
up, the ordinary invariants of all the 4-manifolds involved in the
coverings~(\ref{sum_covering}) can be computed from those of~$X$,
the~$Y(j)$ and~$Z$.

Let us turn towards the equivariant invariants. It has been shown in
Theorem~\ref{thm:relations} that the equivariant invariant~$m_G(X)$
of~$X$ is determined by the ordinary invariants of~$X$ and the~$Y(j)$.
Hence the only thing which has not been determined yet is the
equivariant invariant~\hbox{$m_G(X\#_G(G\times Z))$}
of~\hbox{$X\#_G(G\times Z)$}.  Of course, there is again the forgetful
map which sends that to the ordinary invariants~\hbox{$m(X\#_G(G\times Z))$}
and all the~$m(Y(j)\#Z)$. By Theorem~\ref{thm:not injective}, this map
is not injective. So at this point one cannot deduce the equivariant
invariant~\hbox{$m_G(X\#_G(G\times Z))$} from the ordinary invariants.
But it can be deduced from the equivariant invariants $m_G(X)$
and~$m_G(G\times Z)$: using the general remarks on equivariant
connected sums above, there is an identification
\begin{displaymath}
  m_G(X\#_G(G\times Z))\cong m_G(X)\wedge m_G(G\times Z).
\end{displaymath} 
As described in Example~\ref{ex:2}, the equivariant invariant
$m_G(G\times Z)$ of $G\times Z$ is given by
\begin{displaymath}
	m_G(G\times Z)\cong\bigwedge_{g\in G}m(Z). 
\end{displaymath} 
The following result summarises the discussion.

\begin{prop}
	In the Galois covering situation~\eqref{sum_covering}, the equivariant 	Bauer--Furuta invariant is given as	
	\begin{displaymath}
  		m_G(X\#_G(G\times Z))\cong m_G(X)\wedge\bigwedge_{g\in G}m(Z),
	\end{displaymath}
	and~$m_G(X)$ is determined by the ordinary invariants~$m(X)$ and~$m(Y(j))$.
\end{prop}

To sum up,~(\ref{sum_covering}) is a situation, where an equivariant extension of the connected sum theorem allows one to work around the difficulties posed by the non-injectivity of the ghost map, so that the equivariant invariant can nevertheless be determined from ordinary invariants. It would be interesting to see other examples where this holds (or not).

%%%%%%%%%%%%%%%%%%%%%%%%%%%%%%%%%%%%%%%%%%%%%%%%%%%%%%%%%%%%%%%%%%

%%%%%%%%%%%%%%%%%%%%%%%%%%%%%%%%%%%%%%%%%%%%%%%%%%%%%%%%%%%%%%%%%%

\vfill

\parbox{\linewidth}{%
Markus Szymik\\
Department of Mathematical Sciences\\
NTNU Norwegian University of Science and Technology\\
7491 Trondheim\\
NORWAY\\
\href{mailto:markus.szymik@ntnu.no}{markus.szymik@ntnu.no}\\
\href{https://folk.ntnu.no/markussz}{folk.ntnu.no/markussz}}

\end{document}